\documentclass[12pt,reqno]{article}

\usepackage[usenames]{color}
\usepackage{amssymb}
\usepackage{graphicx}
\usepackage{amscd}

\usepackage[colorlinks=false]{hyperref}

\definecolor{webgreen}{rgb}{0,.5,0}
\definecolor{webbrown}{rgb}{.6,0,0}

\usepackage{color}
\usepackage{fullpage}
\usepackage{float}

\usepackage{graphics,amsmath,amssymb}
\usepackage{amsthm}
\usepackage{amsfonts}
\usepackage{latexsym}
\usepackage{epsf}

\newcommand{\seqnum}[1]{\href{http://oeis.org/#1}{\underline{#1}}}

\begin{document}

\begin{center}
\vskip 0.5cm
\end{center}

\theoremstyle{plain}
\newtheorem{theorem}{Theorem}
\newtheorem{corollary}[theorem]{Corollary}
\newtheorem{lemma}[theorem]{Lemma}
\newtheorem{proposition}[theorem]{Proposition}

\theoremstyle{definition}
\newtheorem{definition}[theorem]{Definition}
\newtheorem{example}[theorem]{Example}
\newtheorem{conjecture}[theorem]{Conjecture}

\theoremstyle{remark}
\newtheorem{remark}[theorem]{Remark}

\begin{center}
\vskip 0.7cm{\LARGE\bf 
Upper Bounds for Prime Gaps \\
\vskip .1in
Related to Firoozbakht's Conjecture
}
\vskip 0.7cm
\large
Alexei Kourbatov\\
www.JavaScripter.net/math\\
15127 NE 24th St., \#578\\
Redmond, WA 98052 \\
USA\\
\href{mailto:akourbatov@gmail.com}{\tt akourbatov@gmail.com}
\end{center}

\vskip .2 in

\begin{abstract}\noindent
We study two kinds of conjectural bounds for the prime gap after the $k$th prime~$p_k$: 
(A) $p_{k+1} < (p_k)^{1+1/k}$ and
(B) $p_{k+1}-p_k < \log^2 p_k - \log p_k - b$ for $k>9$.
The upper bound (A) is equivalent to Firoozbakht's conjecture. 
We prove that (A) implies (B) with $b=1$; on the other hand, (B) with $b=1.17$ implies (A). 
We also give other sufficient conditions for (A) that have the form 
(B) with $b\to1$ as $k\to\infty$.
\end{abstract}

\section{Introduction}
In 1982 Firoozbakht proposed the following conjecture \cite[p.\,185]{ribenboim}:

\medskip\noindent
{\bf Firoozbakht's Conjecture.}
If $p_k$ is the $k$th prime, the sequence $(p_k^{1/k})_{k\in{\mathbb N}}$ is decreasing.

\smallskip\noindent
Equivalently, for all $k\ge1$, the prime $p_{k+1}$ is bounded by the inequality 
\begin{equation}\label{ineqf}
p_{k+1} < (p_k)^{1+1/k}.
\end{equation}
Several authors \cite{rivera,sinha2010,zwsun2013,zwsun2012} have observed that
\begin{itemize}
 \item Firoozbakht's conjecture (\ref{ineqf}) 
       implies {\it Cram\'er's conjecture} $p_{k+1}-p_k=O(\log^2{\negthinspace}p_k)$ \cite{cramer}.
 \item If conjecture (\ref{ineqf}) is true and $k$ is large, then
 \begin{equation}\label{ineqln}
 p_{k+1}-p_k  ~<~ \log^2 p_k - \log p_k.
 \end{equation}
\end{itemize}
(Sun \cite{zwsun2013,zwsun2012} gives a variant of (\ref{ineqln}) with a larger right-hand side,
$\log^2 p_k - \log p_k + 1$.)
 
In Section \ref{sec2} we prove that (\ref{ineqf}) implies a sharper bound than (\ref{ineqln}):
\begin{equation}\label{in2a}
p_{k+1}-p_k  ~<~ \log^2 p_k - \log p_k - b  
\quad\mbox { for all } k>9, 
\end{equation}
with $b=1$.
If the exact value of $k=\pi(p_k)$ is not available, then a violation of (\ref{ineqln}) 
or (\ref{in2a}) might be used to {\it disprove} Firoozbakht's conjecture (\ref{ineqf}). 
However, given a pair of primes $p_k$, $p_{k+1}$, 
the validity of (\ref{ineqln}) alone is not enough for the verification of (\ref{ineqf}). 
We discuss this in more detail in Section \ref{sec3}; see also \cite{kourbatov2015v}.
In Section \ref{sec4} we prove that (\ref{in2a}) with $b=1.17$ implies (\ref{ineqf});
we also give other sufficient conditions for (\ref{ineqf}).
Probabilistic considerations \cite[OEIS \seqnum{A235402}]{cramer,kourbatov2014,kourbatov2015v}
suggest that bounds (\ref{ineqf}), (\ref{ineqln}), (\ref{in2a}) hold for {\it almost all maximal gaps} between primes.

\section{A corollary of Firoozbakht's conjecture}\label{sec2}

\begin{theorem}\label{th1}
If conjecture $(\ref{ineqf})$ is true, then 
$$
p_{k+1}-p_k  ~<~ \log^2 p_k - \log p_k - 1  \quad \mbox { for all } k>9. 
$$
\end{theorem}

\begin{proof}
It is easy to check that
\begin{equation}\label{in4}
{x+\log^2 x \over \log x-1-{1\over\log x}} ~<~ {x\over\log x-1-{1\over\log x}-{1\over\log^2 x}}
\ \ \mbox{ for }x\ge285967. 
\end{equation}
Denote by $\pi(x)$ the prime-counting function. From Axler \cite[Corollary 3.6]{axler} we have
\begin{equation}\label{in5}
{x\over\log x-1-{1\over\log x}-{1\over\log^2 x}} ~<~ \pi(x)
\quad\mbox{ for }x\ge1772201.
\end{equation}
Taking the log of both sides of (\ref{ineqf}) we find that Firoozbakht's conjecture (\ref{ineqf}) is equivalent to
\begin{equation}\label{in6}
k ~<~ {\log p_k \over \log p_{k+1}-\log p_k}. 
\end{equation}
Let $k\ge133115$. Then $p_k\ge1772201$.
By setting $x=p_k$ in (\ref{in4}) and (\ref{in5}), we see that 
inequalities (\ref{in4}), (\ref{in5}), (\ref{in6}) form a chain. 
Therefore, if Firoozbakht's conjecture is true, then
\begin{equation}\label{in7}
{p_k+\log^2 p_k \over \log p_k - 1 -{1\over\log p_k}} ~<~ {\log p_k \over \log p_{k+1}-\log p_k} 
\quad\mbox{ for }p_k\ge1772201.
\end{equation}
Cross-multiplying, we get
\begin{equation}\label{in8}
(\log p_{k+1}-\log p_k)(p_k+\log^2 p_k) ~<~ \log^2 p_k - \log p_k - 1.
\end{equation}
We have
\begin{equation}\label{in9}
{y \over x+y} ~<~ \log(x+y)-\log x \qquad\mbox{ for every } x,y>0. 
\end{equation}
Setting $x=p_k$ and $y=p_{k+1}-p_k$, we can replace the left-hand side of (\ref{in8}) 
by a smaller quantity $(p_{k+1}-p_k)(p_k+\log^2 p_k)/p_{k+1}$ to obtain the inequality
$$ 
{(p_{k+1}-p_k)(p_k+\log^2 p_k)\over p_{k+1}} ~<~ \log^2 p_k - \log p_k - 1,
$$ 
which is equivalent to
$$
(p_{k+1}-p_k)(p_k+\log^2 p_k) ~<~ (p_k + (p_{k+1}-p_k))(\log^2 p_k - \log p_k - 1),
$$
$$
p_{k+1}-p_k ~<~ {p_k\over p_k+\log p_k + 1}(\log^2 p_k - \log p_k - 1).
$$
This proves the theorem for every $k\ge133115$ because $p_k/(p_k+\log p_k + 1)<1$.
Separately, for $9<k<133115$ we verify the desired inequality by direct computation. 
\end{proof}

\section{Does a given prime gap confirm or disprove \\ Firoozbakht's conjecture?}\label{sec3}

Given $p_k$ and $p_{k+1}$, where the prime gap $p_{k+1}-p_k$ is ``large'' and
$k=\pi(p_k)$ is {\it not} known, can we decide whether this gap 
confirms or disproves Firoozbakht's conjecture? The answer is, 
in most cases, {\it yes}. We showed this in \cite[Sect.\,3]{kourbatov2015v} 
and established the following theorem:

\begin{theorem}\label{th2} {\rm(\cite[Sect.\,4]{kourbatov2015v}).}
Firoozbakht's conjecture $(\ref{ineqf})$ is true for all primes 
$p_k  < 4\times10^{18}$. \ 
\end{theorem}
In the verification of (\ref{ineqf}) for $p_k  < 4\times10^{18}$
we have {\it not} used bound (\ref{ineqln}) or (\ref{in2a}); see \cite{kourbatov2015v}.
Indeed, (\ref{ineqln}) is a corollary of (\ref{ineqf}); as such,
(\ref{ineqln}) might be true even when (\ref{ineqf}) is false.
Here is a more detailed discussion.  Define (see Table 1):
\begin{eqnarray*}
f_k &=& p_k^{1+1/k}- p_k \mspace{47mu}\mbox{(the upper bound for $p_{k+1}-p_k$ predicted by (\ref{ineqf}));} \\
\ell_k &=& \log^2 p_k - \log p_k    \quad\mbox{(the upper bound for $p_{k+1}-p_k$ predicted by (\ref{ineqln})).}
\end{eqnarray*}
One can prove that $f_k<\ell_k$ when $k\to\infty$; moreover, $f_k=\ell_k-1+o(1)$ (see {\em Appendix}).
Computation shows that $f_k<\ell_k$ for $p_k\ge11783$ \ ($k\ge1412$).
Suppose there is a prime $q \in [p_k+f_k,\ p_k+\ell_k]$; for example, 
there is such a prime, $q=2010929$, when $p_k=2010733$ (see line 7 in Table 1).
Now what if there were no other primes between $p_k$ and $q$?
Then we would have $p_{k+1}=q$, Firoozbakht's conjecture (\ref{ineqf}) would be {\it false}, 
while (\ref{ineqln}) would still be {\it true}.
So (\ref{ineqln}) is not particularly useful for {\it verifying} (\ref{ineqf}).
On the other hand, any violation of (\ref{ineqln}) would immediately disprove
Firoozbakht's conjecture (\ref{ineqf}). Clearly, similar reasoning is valid for (\ref{in2a}) with $b\le1$.
However, in the next section we prove that (\ref{in2a}) with $b=1.17$ is a {\it sufficient condition} for 
Firoozbakht's conjecture (\ref{ineqf}). We will also give a few other sufficient conditions
that have the form (\ref{in2a}) with $b\to1$ as $k\to\infty$.

\begin{center}
\small{
\begin{tabular}{rrrrr}
\hline
  $\large{\vphantom{1^{1^1}}} k$ & \raisebox{.3mm}{$p_k$} & \ \ \raisebox{.3mm}{$p_{k+1}-p_k$}
  & \ \ $f_k=p_k^{1+1/k}\negthinspace-p_k$ & \ $\ell_k = \log^2 p_k - \log p_k $  \\
[0.5ex]\hline
\vphantom{\fbox{$1^1$}}
              6 &               13 &    4 \phantom{1}&    6.934 \phantom{111}&    4.014 \phantom{1111}\\
              9 &               23 &    6 \phantom{1}&    9.586 \phantom{111}&    6.696 \phantom{1111}\\ 
             30 &              113 &   14 \phantom{1}&   19.286 \phantom{111}&   17.621 \phantom{1111}\\ 
            217 &             1327 &   34 \phantom{1}&   44.709 \phantom{111}&   44.515 \phantom{1111}\\ 
           3385 &            31397 &   72 \phantom{1}&   96.188 \phantom{111}&   96.861 \phantom{1111}\\ 
          31545 &           370261 &  112 \phantom{1}&  150.529 \phantom{111}&  151.581 \phantom{1111}\\ 
         149689 &          2010733 &  148 \phantom{1}&  194.972 \phantom{111}&  196.142 \phantom{1111}\\ 
        1319945 &         20831323 &  210 \phantom{1}&  265.959 \phantom{111}&  267.137 \phantom{1111}\\  
     1094330259 &      25056082087 &  456 \phantom{1}&  548.237 \phantom{111}&  549.389 \phantom{1111}\\ 
    94906079600 &    2614941710599 &  652 \phantom{1}&  787.801 \phantom{111}&  788.925 \phantom{1111}\\  
   662221289043 &   19581334192423 &  766 \phantom{1}&  904.982 \phantom{111}&  906.097 \phantom{1111}\\  
  6822667965940 &  218209405436543 &  906 \phantom{1}& 1055.966 \phantom{111}& 1057.071 \phantom{1111}\\  
 49749629143526 & 1693182318746371 & 1132 \phantom{1}& 1193.418 \phantom{111}& 1194.516 \phantom{1111}\\   
\hline
\end{tabular}
Table 1: Upper bounds for prime gaps
$p_{k+1}-p_k{\vphantom{1^{1^{1^1}}}}$ predicted by (1) and (2); \ \ 
$p_k \in \mbox{\seqnum{A111943}}$ \cite{oeis} \\ [0.5em]
}
\end{center}

\section{Sufficient conditions for Firoozbakht's conjecture}\label{sec4}

\begin{theorem}\label{th3}
If
\begin{equation}\label{in10}
p_{k+1}-p_k  ~<~ \log^2 p_k - \log p_k - 1.17  \quad \mbox { for all } k>9 \ \ (p_k\ge29), 
\end{equation}
then Firoozbakht's conjecture $(\ref{ineqf})$ is true.
\end{theorem}

\begin{proof}
From Axler \cite[Corollary 3.5]{axler} (see {\it Corrigendum} \ref{corrig}) we have
\begin{equation}\label{in11}
\log x - 1 - {1.17\over\log x} ~<~ {x\over\pi(x)}
\quad\mbox{ for every } x\ge 2634800823.     
\end{equation}
Multiplying both sides of (\ref{in11}) by $\log x$, taking $x=p_k$, 
and using (\ref{in10}), we get
\begin{equation}\label{in12}
p_{k+1}-p_k ~<~ \log^2 p_k - \log p_k - 1.17 ~<~ {p_k\log p_k\over k};
\end{equation}
therefore,
\begin{equation}\label{in13}
{p_{k+1}-p_k \over p_k} ~<~ {\log p_k\over k}. 
\end{equation}
We have
$$
\log(x+y)-\log x ~<~ {y \over x}
\qquad\mbox{ for every } x,y>0.
$$
Setting $x=p_k$ and $y=p_{k+1}-p_k$, we can replace the left-hand side of (\ref{in13}) 
by a smaller quantity $\log p_{k+1}-\log p_k$ to obtain the inequality
\begin{equation}\label{in14}
\log p_{k+1}-\log p_k  ~<~ {\log p_k\over k},  
\end{equation}
which is equivalent to
$$
\log_{p_k} {p_{k+1}\over p_k} ~<~ {1\over k}.
$$
Now, exponentiation with base $p_k$ yields (\ref{ineqf}) for $p_k\ge 2634800823$.
This completes the proof since for $p_k\in[29,2634800823]$ both (\ref{ineqf}) and (\ref{in10})
hold unconditionally. 
\end{proof}

\smallskip\noindent
{\bf Other sufficient conditions for (\ref{ineqf}). }
Based on the $\pi(x)$ formula of Panaitopol \cite{panaitopol}, 
Axler gives a family of upper bounds for $\pi(x)$ \cite[Corollary 3.5]{axler}:
\begin{eqnarray*}
\pi(x) &<& {x\over\log x - 1 - {1.17\over\log x}}
     \mspace{190mu} \mbox{ for } x\ge2634800823, \\
\pi(x) &<& {x\over\log x - 1 - {1\over\log x} - {3.83\over\log^2 x}}
     \mspace{133mu} \mbox{ for } x\ge9.25, \\
\pi(x) &<& {x\over\log x - 1 - {1\over\log x} - {3.35\over\log^2 x} - {15.43\over\log^3 x}}
     \mspace{75mu}  \mbox{ for } x\ge14.36, \\
\pi(x) &<& {x\over\log x - 1 - {1\over\log x} - {3.35\over\log^2 x} - {12.65\over\log^3 x} - {89.6\over\log^4 x}}
     \quad \mbox{ for } x\ge21.95. \\
\end{eqnarray*}
Just as in Theorem \ref{th3}, we can transform the above upper bounds into 
sufficient conditions for Firoozbakht's conjecture (\ref{ineqf}) and obtain our next theorem. 
\begin{theorem}\label{th4} 
If one or more of the following conditions hold for all $p_k>4\times10^{18}:$
\begin{eqnarray*}
p_{k+1}-p_k &<& \log^2 p_k - \log p_k - 1.17,  \\
p_{k+1}-p_k &<& \log^2 p_k - \log p_k - 1 - {3.83\over\log p_k}, \phantom{1^{1^{1^1}}\over1} \\
p_{k+1}-p_k &<& \log^2 p_k - \log p_k - 1 - {3.35\over\log p_k} - {15.43\over\log^2 p_k},    \\
p_{k+1}-p_k &<& \log^2 p_k - \log p_k - 1 - {3.35\over\log p_k} - {12.65\over\log^2 p_k} - {89.6\over\log^3 p_k}, 
\end{eqnarray*}
then Firoozbakht's conjecture $(\ref{ineqf})$ is true.
\end{theorem}

In the statement of Theorem \ref{th4}, we have taken into account that for $p_k<4\times10^{18}$
conjecture $(\ref{ineqf})$ holds unconditionally \cite{kourbatov2015v}.
We do not give a proof of Theorem \ref{th4}; it is fully similar to the proof of Theorem \ref{th3}.

\medskip\noindent
{\bf Remarks.} 

(i) In inequality (\ref{in10}) the right-hand side is an increasing function of $p_k$.
Therefore, if (\ref{in10}) holds for a {\it maximal prime gap} with
$p_k = \mbox{\seqnum{A002386}}(n)$, then (\ref{in10}) must also be true 
for {\em every} $p_k$ between $\mbox{\seqnum{A002386}}(n)$ and $\mbox{\seqnum{A002386}}(n+1)$.
So an easy way to prove Theorem \ref{th2} is to 
check (\ref{ineqf}) directly for all primes $p_k\le89$, then verify (\ref{in10}) 
just for maximal prime gaps with $p_k=\mbox{\seqnum{A002386}}(n)\ge89$. 

(ii) In Theorem \ref{th4}, the coefficients of $(\log p_k)^{-n}$ approximate  
the terms of OEIS sequence \seqnum{A233824}: a recurrent sequence in Panaitopol's formula for $\pi(x)$
\cite{panaitopol}.

\section{Appendix: An asymptotic formula for $p_k^{1+1/k}-p_k$} 

\begin{theorem}\label{th5} 
Let $p_k$ be the $k$-th prime, and let $f_k = p_k^{1+1/k}- p_k$, then 
$$
f_k = \log^2 p_k - \log p_k - 1 +o(1) \quad\mbox{ as } k\to\infty 
\qquad\mbox{\rm(cf.~OEIS \seqnum{A246778}).}
$$
\end{theorem} 

\begin{proof}
From Axler \cite[Corollaries 3.5, 3.6]{axler} we have
\begin{equation}\label{dblin}
{x\over\log x-1-{1\over\log x}-{1\over\log^2 x}} ~<~ \pi(x) ~<~  
{x\over\log x-1-{1\over\log x}-{3.83\over\log^2 x}} \quad\mbox{ for } x\ge1772201.
\end{equation}  
By definition of $f_k$, we have $\log_{p_k}(p_k+f_k) = 1+1/k$, so 
$\displaystyle
k = \pi(p_k) = {\log p_k \over \log(p_k+f_k)-\log p_k}.
$ 
Therefore, for $x=p_k\ge1772201$, we can rewrite (\ref{dblin}) as
\begin{equation}\label{dblin2}
{p_k\over\log p_k-1-{1\over\log p_k}-{1\over\log^2 p_k}} ~<~ {\log p_k \over \log(p_k+f_k)-\log p_k} ~<~  
{p_k\over\log p_k-1-{1\over\log p_k}-{3.83\over\log^2 p_k}}.
\end{equation}
{\em An upper bound for} $f_k$.
We combine (\ref{in4}) with the left inequality of (\ref{dblin2}) to get
\begin{equation}\label{inA}
{p_k+\log^2 p_k \over \log p_k - 1 -{1\over\log p_k}} ~<~ {\log p_k \over \log(p_k + f_k)-\log p_k} 
\quad\mbox{ for }p_k\ge1772201.
\end{equation}
Cross-multiplying and using (\ref{in9}), similar to Theorem \ref{th1}, we obtain
$$
{f_k(p_k+\log^2 p_k)\over p_k + f_k} ~<~ {(\log(p_k + f_k)-\log p_k)(p_k+\log^2 p_k)} ~<~ \log^2 p_k - \log p_k - 1,
$$
$$
f_k(p_k+\log^2 p_k) ~<~ (p_k + f_k)(\log^2 p_k - \log p_k - 1),
$$
$$
f_k ~<~ {p_k\over p_k+\log p_k + 1}(\log^2 p_k - \log p_k - 1)
    ~<~ \log^2 p_k - \log p_k - 1.
$$
{\em A lower bound for} $f_k$.
From the right inequality of (\ref{dblin2}) we get
$$
{\log^2 p_k - \log p_k - 1 - {3.83\over\log p_k}\over p_k}
~<~ \log(p_k + f_k)-\log p_k ~<~ {f_k\over p_k},
$$
$$
\log^2 p_k - \log p_k - 1 - {3.83\over\log p_k} ~<~ f_k.
$$
Together, the upper and lower bounds yield the desired asymptotic formula for $k\to\infty$.
\end{proof}

\section{Acknowledgments}
The author expresses his gratitude to the anonymous referee for numerous useful suggestions,
and to all contributors and editors 
of the websites {\it OEIS.org} and {\it PrimePuzzles.net}, particularly to Farideh Firoozbakht 
for proposing a very interesting conjecture. 
Thanks are also due to Christian Axler for proving the $\pi(x)$ bounds \cite{axler}
used in Theorems \ref{th1}, \ref{th3}--\ref{th5}.

\pagebreak

{\small

\bigskip
\hrule
\bigskip

\noindent 2010 {\it Mathematics Subject Classification}: 11N05.

\noindent \emph{Keywords: } 
Cram\'er conjecture, Firoozbakht conjecture, prime gap.

\bigskip
\hrule
\bigskip

\noindent (Concerned with sequences 
 \seqnum{A002386},
 \seqnum{A005250},
 \seqnum{A111943},
 \seqnum{A182134},
 \seqnum{A182514},
 \seqnum{A182519},
 \seqnum{A205827},
 \seqnum{A233824},
 \seqnum{A235402},
 \seqnum{A235492},
 \seqnum{A245396},
 \seqnum{A246776},
 \seqnum{A246777},
 \seqnum{A246778},
 \seqnum{A246810},
 \seqnum{A249669}.)

\bigskip
\hrule
\bigskip
}

\section{Corrigendum}\label{corrig}
The proof of Theorem \ref{th3} [K], as well as subsequent discussion, 
should reflect the true range of applicability of (\ref{in11}), necessitating
the following changes (see [A]):

\medskip\noindent
In inequality (\ref{in11}), replace ``$x\ge5.43$'' with ``$x\ge 2634800823$''

\medskip\noindent
Remove ``Let $k>9$.'' after inequality (\ref{in11}). 

\medskip\noindent
In inequalities (\ref{in12}) and (\ref{in13}), remove ``for $p_k\ge29$''.

\medskip\noindent
Replace the last two sentences of the proof of Theorem \ref{th3} with

\smallskip\noindent{\footnotesize
Now, exponentiation with base $p_k$ yields (\ref{ineqf}) for $p_k\ge 2634800823$.
This completes the proof since for $p_k\in[29,2634800823]$ both (\ref{ineqf}) and (\ref{in10})
hold unconditionally. 
}

\medskip\noindent
In the 2nd display formula on p.\,5, replace ``$x\ge5.43$'' with ``$x\ge 2634800823$''

\medskip\noindent
{\it These changes have been incorporated in\,} arXiv:1506.03042v4.

\medskip\noindent
{\bf References.}

\smallskip\noindent
[K] A.~Kourbatov, Upper bounds for prime gaps related to Firoozbakht's conjecture,
{\it Journal of Integer Sequences}, {\bf 18} (2015), Article 15.11.2. 

\smallskip\noindent
[A] C.~Axler, Corrigendum to ``New  bounds for the prime counting function'',
{\it Integers} {\bf 16} (2016), A22, 15 pp. 
\url{http://math.colgate.edu/~integers/vol16.html}

\end{document}